\documentclass[10pt]{amsart}
\usepackage{enumerate,amsmath,amssymb,latexsym,
amsfonts, amsthm, amscd, MnSymbol}

\theoremstyle{plain}

\newtheorem{theorem}{Theorem}

\newtheorem*{theorem*}{Theorem}

\numberwithin{equation}{section}

\newcommand{\R}{\mathbb{R}}
\newcommand{\C}{\mathbb{C}}

\newcommand{\ii}{{\rm i}}

\newcommand{\ppi}{\varpi}

\newcommand{\rn}{{\rm rn}_2}

\begin{document}

\title {A root elliptic function in signature four}

\date{}

\author[P.L. Robinson]{P.L. Robinson}

\address{Department of Mathematics \\ University of Florida \\ Gainesville FL 32611  USA }

\email[]{paulr@ufl.edu}

 \keywords{}

\begin{abstract}
We present a new approach to elliptic functions in signature four, offering a fresh perspective on work of Li-Chien Shen. 
\end{abstract}

\maketitle

\section*{Introduction} 

\medbreak 

By modifying an approach to the classical elliptic functions ${\rm sn}, {\rm cn}$ and ${\rm dn}$ of Jacobi, Shen [2014] introduced analogous functions ${\rm sn}_2, {\rm cn}_2$ and ${\rm dn}_2$ into the Ramanujan theory of elliptic functions in signature four. Naturally, the parallel between the classical functions ${\rm sn}, {\rm cn}, {\rm dn}$ and the signature-four functions ${\rm sn}_2, {\rm cn}_2, {\rm dn}_2$ is not perfect. On the one hand, the signature-four functions satisfy algebraic equations of the familiar forms ${\rm cn}_2^2 + {\rm sn}_2^2 = 1$ and ${\rm dn}_2^2 + \kappa^2 {\rm sn}_2^2 = 1$ (in which $0 < \kappa < 1$ is the modulus). On the other hand, they satisfy differential equations of other forms: for example, the classical differential equation $({\rm dn}')^2 = (1 - {\rm dn}^2) ({\rm dn}^2 - \lambda^2)$ is replaced by $({\rm dn}_2')^2 = 2 (1 - {\rm dn}_2)({\rm dn}_2^2 - \lambda^2)$ (in which $\lambda = (1 - \kappa^2)^{1/2}$ is the complementary modulus). A further deviation from the classical theory is the fact that, whereas the even functions ${\rm dn}_2$ and ${\rm cn}_2$ are elliptic, the odd function ${\rm sn}_2$ is not. The approach in [2014] is based primarily on the specific function ${\rm dn}_2$. Here, we reconsider the approach and suggest an alternative elliptic function on which to base the signature-four theory. This alternative elliptic function is odd and leads easily to the functions ${\rm sn}_2, {\rm cn}_2$ and ${\rm dn}_2$ (along with their properties, including the non-elliptic nature of the first-named). 

\medbreak 

We begin by recounting briefly how the functions ${\rm sn}_2, {\rm cn}_2$ and ${\rm dn}_2$ are introduced in [2014]. As in the classical Jacobian theory, we fix the modulus $\kappa \in (0, 1)$ and the complementary modulus $\lambda = (1 - \kappa^2)^{1/2} \in (0, 1)$. The (acute) modular angle $\alpha$ is then specified by $\kappa = \sin \alpha$ so that also $\lambda = \cos \alpha$. 

\medbreak 

The function 
$$\R \to \R: T \mapsto \int_0^T F(\tfrac{1}{4}, \tfrac{3}{4} ; \tfrac{1}{2}; \kappa^2 \sin^2 t) \, {\rm d}t$$
is an odd strictly-increasing bijection, whose inverse we denote by $\phi$: thus, if $u \in \R$ then 
$$u = \int_0^{\phi (u)} F(\tfrac{1}{4}, \tfrac{3}{4} ; \tfrac{1}{2}; \kappa^2 \sin^2 t) \, {\rm d}t.$$
The composite 
$$\psi = \arcsin (\kappa \sin \phi)$$
is an auxiliary function on $\R$ that takes its values in the interval $[- \alpha, \alpha] \subseteq (- \tfrac{1}{2} \pi, \tfrac{1}{2} \pi)$. 

\medbreak 

Now, ${\rm sn}_2, {\rm cn}_2$ and ${\rm dn}_2$ are introduced as the composites 
$${\rm sn}_2 = \sin \phi, \, {\rm cn}_2 = \cos \phi \; \; {\rm and} \; \; {\rm dn}_2 = \cos \psi$$ 
defined initially on $\R$ as domain; it is then established that ${\rm dn}_2$ extends to the complex plane as an elliptic function. The whole development in [2014] proceeds from this point: expressions are found for ${\rm dn}_2$ and its companions in terms of Weierstrass $\wp$-functions, in terms of Jacobian elliptic functions and in terms of theta functions; and results of Ramanujan relating the hypergeometric function $F(\tfrac{1}{4}, \tfrac{3}{4} ; 1; \bullet)$ to the Eisenstein series $E_4$ and $E_6$ are recovered. 

\medbreak 

We propose to approach the signature-four elliptic function theory along similar lines but from a different starting point. Specifically, rather than base our development on the function $\cos \psi$ we shall instead base it on the function $\sin \tfrac{1}{2} \psi$. As we shall see, this extends to an elliptic function with properties that are closer to those of the classical Jacobian elliptic functions: thus, it has simple zeros and simple poles; also, the first-order differential equation that it satisfies involves a quartic rather than a cubic. As it is odd, it perhaps serves partly to fill the void that is left by the non-elliptic function ${\rm sn}_2$ as a replacement for the elliptic function ${\rm sn}$. We shall write $\rn$ for the function $\sin \tfrac{1}{2} \psi$ and for its elliptic extension: the use of `r' as a prefix hints at the r\^ole of $\rn$ as a root on which to base the elliptic functions; the proximity of `r' to `s' suggests the kinship of $\rn$ to the sine function. 

\medbreak 

\section*{The root elliptic function $\rn$}

\medbreak 

With the function $\phi : \R \to \R$ defined by the rule that if $u \in \R$ then 
$$u = \int_0^{\phi (u)} F(\tfrac{1}{4}, \tfrac{3}{4} ; \tfrac{1}{2}; \kappa^2 \sin^2 t) \, {\rm d}t$$
and with the auxiliary function $\psi: \R \to \R$ defined by 
$$\psi = \arcsin (\kappa \sin \phi)$$
as above, we introduce the function $\rn : \R \to \R$ by the definition 
$$\rn: = \sin \tfrac{1}{2} \psi.$$
The functions $\phi$ and $\psi$ are plainly odd. Recall that $\phi$ is a strictly-increasing bijection; we claim that the function $\psi$ is periodic. To see this, let us introduce the `complete integral'
$$K = \int_0^{\frac{1}{2} \pi}  F(\tfrac{1}{4}, \tfrac{3}{4} ; \tfrac{1}{2}; \kappa^2 \sin^2 t) \, {\rm d}t$$
and note that hypergeometric expansion and termwise integration yield 
$$K = \tfrac{1}{2} \pi \, F(\tfrac{1}{4}, \tfrac{3}{4} ; 1; \kappa^2).$$ 

\bigbreak 

\begin{theorem} \label{4K}
The function $\rn: \R \to \R$ has period $4K.$  
\end{theorem} 

\begin{proof} 
Note first that 
$$2 K = \int_0^{\pi}  F(\tfrac{1}{4}, \tfrac{3}{4} ; \tfrac{1}{2}; \kappa^2 \sin^2 t) \, {\rm d}t.$$
With $u \in \R$ write $U = \phi(u)$. Integrate the function 
$$t \mapsto F(\tfrac{1}{4}, \tfrac{3}{4} ; \tfrac{1}{2}; \kappa^2 \sin^2 t)$$ 
over the interval $[0, \pi + U]$ divided at $\pi$: the integral over $[0, \pi]$ is exactly $2 K$; the integral over $[\pi, \pi + U]$ becomes exactly $u$ after shifting the integration variable by $\pi$. This proves that 
$$\phi (u + 2 K) = \phi(u) + \pi.$$
As the sine function merely reverses its sign when its argument is increased by $\pi$ and as the function $\arcsin$ is odd, it follows that 
$$\psi( u + 2 K) = - \psi(u).$$
Finally, 
$$\rn(u + 2 K) = - \rn(u).$$
\end{proof} 

\medbreak 

Evidently, $4 K$ is the least positive period of $\rn$. 

\medbreak 

We now turn our attention to the elliptic extendibility of the function $\rn$, which we establish by verifying that $\rn$ satisfies a first-order initial value problem whose solutions are known to be elliptic. For this purpose, we shall require the following hypergeometric evaluation; although it is a standard result, we take this opportunity to present a proof. 

\medbreak 

\begin{theorem} \label{hyper}
$$F(\tfrac{1}{4}, \tfrac{3}{4} ; \tfrac{1}{2}; \sin^2 \psi) = \frac{\cos \frac{1}{2} \psi}{\cos \psi}.$$
\end{theorem} 

\begin{proof} 
Here, $\psi$ may be either an acute angle or the auxiliary function $\arcsin (\kappa \sin \phi)$. This result is a special case of item (11) on page 101 in [1953] and we shall establish it in two stages. The first stage is to note that if $z \in \mathbb{D}$ is any point in the open unit disc then 
$$F(\tfrac{1}{4}, \tfrac{3}{4} ; \tfrac{1}{2}; z^2) = \tfrac{1}{2} \, \big[(1 + z)^{-1/2} + (1 - z)^{-1/2}\big];$$
this itself is a special case of item (6) on page 101 in [1953] and is readily verified by simply expanding the hypergeometric series on the left and the two geometric series on the right. The second stage is to note that 
$$\tfrac{1}{2} \, \big[(1 + \sin \psi)^{-1/2} + (1 - \sin \psi)^{-1/2}\big] = \frac{\cos \frac{1}{2} \psi}{\cos \psi};$$
this trigonometric identity may be safely left as an easy exercise.  
\end{proof} 

\medbreak 

We remark that this is essentially a special case of a result recorded by Ramanujan in his second notebook: take $n = 1/4$ in Entry 35(iii) on page 99 of [1989]. 

\medbreak 

The initial value problem satisfied by $\rn$ is as follows. 

\medbreak 

\begin{theorem} \label{IVP}
The function $\rn: \R \to \R$ satisfies the initial condition $\rn (0) = 0$ and the differential equation 
$$(\rn ')^2 = \rn^4 - \rn^2 + \tfrac{1}{4} \kappa^2.$$
\end{theorem} 

\begin{proof} 
Differentiate. From 
$$u = \int_0^{\phi (u)} F(\tfrac{1}{4}, \tfrac{3}{4} ; \tfrac{1}{2}; \kappa^2 \sin^2 t) \, {\rm d}t$$
and Theorem \ref{hyper} there follows 
$$\phi ' = \frac{\cos \psi}{\cos \frac{1}{2} \psi}.$$ 
From $\sin \psi = \kappa \sin \phi$ there follows 
$$\psi ' = \kappa \frac{\cos \phi}{\cos \psi} \phi'.$$ 
From the definition $\rn = \sin \tfrac{1}{2} \psi$ there follows 
$$\rn ' = \tfrac{1}{2} \cos \tfrac{1}{2} \psi \cdot \psi '.$$ 
Taking these three facts together, mass cancellation results in 
$$\rn ' = \tfrac{1}{2} \kappa \cos \phi.$$ 
Finally, squaring and trigonometric duplication yield 
$$4 (\rn ')^2 = \kappa^2 - \sin^2 \psi = \kappa^2 - 4 \sin^2 \tfrac{1}{2} \psi \cos^2 \tfrac{1}{2} \psi = \kappa^2 - 4 \rn^2 (1 - \rn^2).$$
\end{proof} 

\medbreak 

As its right-hand side is a quartic with simple zeros, this differential equation has solutions that are elliptic in the plane. As the initial value of $\rn$ is not a zero of this quartic, the extraction of $\rn$ in Weierstrassian terms from Theorem \ref{IVP} is a little more involved than is the corresponding extraction of ${\rm dn}_2$ in [2014]. To be explicit, our extraction uses the following result of Example 2 on page 454 of [1927] (to which we refer for its attribution to Weierstrass). 

\medbreak 

Let $f$ be a quartic polynomial with quadrinvariant $g_2$ and cubinvariant $g_3$: say  
$$f(z) = a_0 z^4 + 4 a_1 z^3 + 6 a_2 z^2 + 4 a_3 z + a_4$$ 
with 
$$g_2 = a_0 a_4 - 4 a_1 a_3 + 3 a_2^2$$
and 
$$g_3 = a_0 a_2 a_4 + 2 a_1 a_2 a_3 - a_2^3 - a_0 a_3^2 - a_1^2 a_4;$$ 
and assume that the zeros of $f$ are simple. Then the initial value problem 
$$(w ')^2 = f(w), \; \; w(0) = a$$
has solutions given by 
$$w = a + \frac{A \, \wp ' + \frac{1}{2} f'(a) [\wp - \frac{1}{24} f''(a)] + \frac{1}{24} f(a) f'''(a)}{2[\wp - \frac{1}{24} f''(a)]^2 - \frac{1}{48} f(a) f'''' (a)}$$ 
where $A$ is a square-root of $f(a)$ and where $\wp = \wp(\bullet ; g_2, g_3)$ is the Weierstrass function with $g_2$ and $g_3$ as its invariants. 

\medbreak 

\begin{theorem} \label{P}
The function $\rn$ is given by 
$$\rn = \frac{\tfrac{1}{4} \kappa P'}{(\tfrac{1}{4} \kappa)^2 - (\tfrac{1}{12} + P)^2} =  \frac{\tfrac{1}{4} \kappa P'}{(\tfrac{1}{4} \kappa - \tfrac{1}{12} - P) (\tfrac{1}{4} \kappa + \tfrac{1}{12} + P)}$$ 
where $P = \wp(\bullet; G_2, G_3)$ is the Weierstrass function with invariants 
$$G_2 = \tfrac{1}{12} (1 + 3 \kappa^2) \; \; and \; \; G_3 = \tfrac{1}{216} (1 - 9 \kappa^2).$$ 
\end{theorem} 

\begin{proof} 
Apply the result of Weierstrass just quoted, with $f(z) = z^4 - z^2 + \tfrac{1}{4} \kappa^2$ and $a = 0$. The square-root $A = - \tfrac{1}{2} \kappa$ of $\tfrac{1}{4} \kappa^2$ is preferred by the fact that $\rn (t) = \sin \tfrac{1}{2} \psi (t)$ is positive when $t > 0$ is small. 
\end{proof} 

\medbreak 

As claimed, $\rn$ extends to $\C$ as an elliptic function, for which we continue the name $\rn$. 

\medbreak 

Of course, the second solution to the initial value problem in Theorem \ref{IVP} is simply $- \rn$. The fact that this initial value problem has just these two solutions may be seen (for example) by a further differentiation: after cancellation of $\rn '$ we arrive at the second-order differential equation 
$$\rn '' = 2 \rn^3 - \rn$$
in which the right-hand side is a polynomial in $\rn$; when augmented by specification of the initial values for $\rn$ and its first derivative, this differential equation has a unique solution (locally, hence globally by the principle of analytic continuation).

\medbreak 

The Weierstrass function $P$ has a fundamental pair of periods $(2 \Omega, 2 \Omega')$ with $\Omega > 0$ and $- \ii \, \Omega ' > 0$: the real period $2 \Omega$ is precisely $4 K$ as in Theorem \ref{4K}; the imaginary period $2 \Omega '$ will be revealed in due course. The first-order equation 
$$(P ')^2 = 4 P^3 - G_2 P - G_3 = 4 P^3 - \tfrac{1}{12} (1 + 3 \kappa^2) P - \tfrac{1}{216} (1 - 9 \kappa^2)$$ 
satisfied by $P$ may be rewritten in factorized form as 
$$(P ')^2 = 4 \, (P - \tfrac{1}{6}) \, (P + \tfrac{1}{12} - \tfrac{1}{4} \kappa) \, (P + \tfrac{1}{12} + \tfrac{1}{4} \kappa)$$ 
so that $P$ has midpoint values 
$$P(\Omega) = \tfrac{1}{6}, \; P(\Omega + \Omega ') = - \tfrac{1}{12} + \tfrac{1}{4} \kappa, \; P(\Omega ') = - \tfrac{1}{12} - \tfrac{1}{4} \kappa.$$

\medbreak 

It is convenient to record the following explicit formula for the square of $\rn$. 

\medbreak 

\begin{theorem} \label{square}
$$\rn^2 = \tfrac{1}{4} \kappa^2 \frac{P - \tfrac{1}{6}}{ (P + \tfrac{1}{12} - \tfrac{1}{4} \kappa) \, (P + \tfrac{1}{12} + \tfrac{1}{4} \kappa)} = \tfrac{1}{4} \kappa^2 \frac{P - \tfrac{1}{6}}{(P + \tfrac{1}{12})^2 -  (\tfrac{1}{4} \kappa)^2}.$$ 
\end{theorem} 

\begin{proof} 
Simply combine the result of Theorem \ref{P} with the (factorized) differential equation that is satisfied by $P$. 
\end{proof} 

\medbreak 

Otherwise said, 
$$\rn^2 = \tfrac{1}{4} \kappa^2 \frac{P - P(\Omega)}{ (P - P(\Omega + \Omega')) \, (P - P(\Omega'))}.$$ 

\bigbreak 

The following circumstance is worth recording here. 

\medbreak 

\begin{theorem} \label{coperiodic}
The functions $\rn$ and $P$ are coperiodic. 
\end{theorem} 

\begin{proof} 
Theorem \ref{P} makes it clear that every period of $P$ is a period of $\rn$. In the opposite direction, let $\varpi$ be a period of $\rn$ and therefore of $\rn^2$. Equate the values of $\rn^2$ at the points $z + \ppi$ and $z$: Theorem \ref{square} yields 
$$\frac{P(z + \ppi) - \tfrac{1}{6}}{(P(z + \ppi) + \tfrac{1}{12})^2 -  (\tfrac{1}{4} \kappa)^2} = \frac{P(z) - \tfrac{1}{6}}{(P(z) + \tfrac{1}{12})^2 -  (\tfrac{1}{4} \kappa)^2}.$$
Let $z$ tend to zero: the right-hand side vanishes, whence so does the left-hand side; this raises two possibilities. The one possibility is that $\ppi$ is a pole of $P$ and therefore a period of $P$. The other possibility is that $P(\ppi) = \tfrac{1}{6}$ and therefore that $\ppi$ is congruent to $\Omega$; but this possibility is dismissed by the fact that $2 \Omega$ is the least positive period of $\rn$. 
\end{proof} 

\medbreak 

\medbreak 

The zeros and poles of $\rn$ are located as follows. 

\medbreak 

\begin{theorem} \label{zeropole}
The function $\rn$ has: simple zeros at $0$ and $\Omega$; simple poles at $\Omega + \Omega'$ and $\Omega'$.  
\end{theorem} 

\begin{proof} 
It is perhaps simplest to inspect the alternative formula displayed after Theorem \ref{square}: the double zeros in the denominator yield simple poles of $\rn$; the double zero in the numerator yields a simple zero of $\rn$; and the final simple zero of $\rn$ arises from cancellation between the double pole in the numerator and the quadruple pole in the denominator. Of course, a full accounting of zeros and poles includes points congruent modulo $(2 \Omega, 2 \Omega')$. 
\end{proof} 

\medbreak 

Thus $\rn$ has two simple zeros and two simple poles in each period parallelogram (translated so that no zeros and poles lie on its perimeter) and so $\rn$ has order two as an elliptic function. 

\medbreak 

We now make explicit the effect on $\rn$ of half-period shifts. 

\medbreak 

\begin{theorem} \label{plusOmega}
$\rn (z + \Omega) = - \rn (z).$ 
\end{theorem} 

\begin{proof} 
Several proofs are possible; here are three, the first two being not entirely unrelated. (i) In our proof of Theorem \ref{4K} we observed that the advertised equality is valid when $z$ lies on the real line; its validity is then propagated throughout the complex plane by the principle of analytic continuation. (ii) The function $z \mapsto \rn (z + \Omega)$ satisfies the initial value problem in Theorem \ref{IVP}: the differential equation is autonomous and we have just seen that $\rn (\Omega) = 0$. As $\rn$ does not have $\Omega$ as a period, it follows that $z \mapsto \rn(z + \Omega)$ is the second solution to the initial value problem; see the discussion following Theorem \ref{P}. (iii) Use the Weierstrassian identity 
$$(P(z + \Omega) - P(\Omega))(P(z) - P(\Omega)) = (P(\Omega' + \Omega) - P(\Omega))(P(\Omega') - P(\Omega))$$
along with the specific evaluation 
$$(P(\Omega' + \Omega) - P(\Omega))(P(\Omega') - P(\Omega)) = \tfrac{1}{4}(- 1 + \kappa) \tfrac{1}{4} (- 1 - \kappa) = (\tfrac{1}{4} \lambda)^2$$
in calculations along lines similar to those followed in the proof of the next Theorem. 
\end{proof} 

\medbreak 

Shifts by the other half-periods situate a pole at the origin, as follows. 

\medbreak 

\begin{theorem} \label{plusOmega'}
$$\rn(z + \Omega') = \frac{\frac{1}{2} P '(z)}{\tfrac{1}{6} - P(z)}.$$
\end{theorem} 

\begin{proof} 
Taking into account the midpoint values of $P$ that are recorded immediately prior to Theorem \ref{square}, the Weierstrassian identity 
$$(P(z + \Omega') - P(\Omega'))(P(z) - P(\Omega') = (P(\Omega + \Omega') - P(\Omega'))(P(\Omega) - P(\Omega'))$$
reads 
$$(P(z + \Omega') - P(\Omega'))(P(z) - P(\Omega')) = (\tfrac{1}{2} \kappa)(\tfrac{1}{4} + \tfrac{1}{4} \kappa) = \tfrac{1}{8} \kappa (1 + \kappa)$$
so that 
$$P(z + \Omega') - P(\Omega') = \frac{\frac{1}{8} \kappa (1 + \kappa)}{P(z) - P(\Omega')}$$ 
whence we derive 
$$P '(z + \Omega ') = - \frac{\frac{1}{8} \kappa (1 + \kappa) P '(z)}{(P(z) - P(\Omega'))^2} .$$
To complete the proof, substitute these last two expressions into the formula of Theorem \ref{P} and rearrange. 
\end{proof} 

\medbreak 

Of course, it follows from Theorem \ref{plusOmega} and Theorem \ref{plusOmega'} that 
$$\rn(z + \Omega + \Omega') = \frac{\frac{1}{2} P '(z)}{P(z) - \tfrac{1}{6}}.$$
The devotee of [1927] will now be pleased to consult Miscellaneous Example 12 on page 457 in case $c = - \tfrac{1}{6}$ and $e = \tfrac{1}{2} \kappa.$ 

\medbreak 

Theorem \ref{plusOmega} tells us that shifting by $\Omega$ converts the function $\rn$ to its negative; in contrast, shifting by $\Omega'$ converts $\rn$ almost to its reciprocal. 

\medbreak 

\begin{theorem} \label{recip}
$\rn (z + \Omega')  \rn (z) = \tfrac{1}{2} \kappa.$ 
\end{theorem} 

\begin{proof} 
The stated identity follows easily from Theorem \ref{P} and Theorem \ref{plusOmega'} by substitution. Alternatively, we may instead play a familiar elliptic game: the $\Omega'$ shift interchanges the (simple) zeros and (simple) poles of the function $\rn$; the product of $\rn$ and its $\Omega'$-shift is thus an elliptic function without poles and so constant. We leave independent confirmation that this constant has value $\tfrac{1}{2} \kappa$ as an exercise. 
\end{proof}

\medbreak 

\section*{The functions ${\rm sn}_2, {\rm cn}_2$ and ${\rm dn}_2$}

\medbreak 

We now turn to the r\^ole of the elliptic function $\rn$ as a source for the full array of functions ${\rm sn}_2, {\rm cn}_2$ and ${\rm dn}_2$ introduced in [2014]. 

\medbreak 

We define the function ${\rm dn}_2$ afresh by the rule 
$${\rm dn}_2 = 1 - 2 \rn^2.$$ 
When restricted to the real line, the function on the right reduces to $1 - 2 \sin^2 \tfrac{1}{2} \psi = \cos \psi$ and so agrees with ${\rm dn}_2$ as defined in [2014] and as recalled at the outset of the present paper; by the principle of analytic continuation, this agreement persists throughout the complex plane, thereby justifying our use of the notation. The function ${\rm dn}_2$ defined by this new rule satisfies the initial condition ${\rm dn}_2 (0) = 1$ and the differential equation 
$$({\rm dn}_2')^2 = 2 (1 - {\rm dn}_2)({\rm dn}_2^2 - \lambda^2)$$
on account of the fact that $\rn$ satisfies the initial value problem displayed in Theorem \ref{IVP}. This enables us to express ${\rm dn}_2$ in terms of its coperiodic Weierstrass function. 

\medbreak 

\begin{theorem} \label{dn2} 
The function ${\rm dn}_2$ is given by 
$${\rm dn}_2 = 1 - \frac{\frac{1}{2} \kappa^2}{\frac{1}{3} + p}$$
where $p = \wp(\bullet; g_2, g_3)$ is the Weierstrass function with invariants 
$$g_2 = \tfrac{1}{3} (3 \lambda^2 + 1) \; \; and \; \; g_3 = \tfrac{1}{27} (9 \lambda^2 - 1).$$ 
\end{theorem} 

\begin{proof} 
See [2014]: extraction of ${\rm dn}_2$ from the initial value problem presented immediately prior to the present Theorem is effected by means of the result derived on page 453 of [1927]. Note that the initial value ${\rm dn}_2 (0) = 1$ is a zero of the cubic $2 (1 - {\rm dn}_2)({\rm dn}_2^2 - \lambda^2)$. 
\end{proof} 

\medbreak 

We may of course also express ${\rm dn}_2$ in terms of the Weierstrass function $P$ on account of Theorem \ref{P}; but its identification in terms of $p$ better suits our purposes. 

\medbreak 

We shall write $(2 \omega, 2 \omega ')$ for the fundamental pair of periods for $p$ such that $\omega > 0$ and $- \ii \omega ' > 0$. We shall soon express these periods directly in terms of the modulus $\kappa$. Theorem \ref{dn2} makes it plain that $p$ and ${\rm dn}_2$ have exactly the same periods. The midpoint values of $p$ are 
$$p(\omega) = \tfrac{1}{6} + \tfrac{1}{2} \lambda, \, p(\omega + \omega ') = \tfrac{1}{6} - \tfrac{1}{2} \lambda, \, p(\omega') = - \tfrac{1}{3};$$ 
${\rm dn}_2$ has midpoint values 
$${\rm dn}_2 (\omega) = \lambda, \; {\rm dn}_2 (\omega + \omega') = - \lambda$$
and a pole at $\omega'$. 

\medbreak 

A comparison between the Weierstrass functions $p$ and $P$ is fruitful. A glance at their invariants $g_2, g_3$ and $G_2, G_3$ makes it clear that they are related. For clarity in what follows, we shall refine our notation just a little by labelling each of these $\wp$-functions with the modulus $\kappa$ from which it was derived. Thus, $P_{\kappa}$ will be the $\wp$-function $P$ of Theorem \ref{P} with quadrinvariant $\tfrac{1}{12} (1 + 3 \kappa^2)$ and cubinvariant $\tfrac{1}{216} (1 - 9 \kappa^2)$ while $p_{\kappa}$ will be the $\wp$-function $p$ of Theorem \ref{dn2} with quadrinvariant $\tfrac{1}{3} (3 \lambda^2 + 1)$ and cubinvariant $\frac{1}{27} (9 \lambda^2 - 1)$; accordingly, $p_{\lambda}$ is the $\wp$-function with quadrinvariant $\tfrac{1}{3} (3 \kappa^2 + 1)$ and cubinvariant $\frac{1}{27} (9 \kappa^2 - 1)$. 

\medbreak 

\begin{theorem} \label{Pp}
$p_{\lambda} (z) = - 2 P_{\kappa} (\ii \sqrt2 z).$ 
\end{theorem} 

\begin{proof} 
We invoke the Weierstrassian homogeneity relation 
$$\wp( z; \gamma^4 G_2, \gamma^6 G_3) = \gamma^2 \wp(\gamma z; G_2, G_3).$$
Let $\gamma = \ii \sqrt2$: from $G_2 =  \tfrac{1}{12} (1 + 3 \kappa^2)$ and $G_3 = \tfrac{1}{216} (1 - 9 \kappa^2)$ we get $\gamma^4 G_2 = \tfrac{1}{3} (3 \kappa^2 + 1)$ and $\gamma^6 G_3 = \frac{1}{27} (9 \kappa^2 - 1)$. With these choices, revisit the homogeneity relation: its left side becomes $p_{\lambda} (z)$; its right side becomes $- 2 P_{\kappa} (\ii \sqrt2 z).$ 

\end{proof} 

\medbreak 

It follows that the period lattice of $P_{\kappa}$ is obtained from the period lattice of $p_{\lambda}$ by a $\sqrt2$ dilation and a quarter rotation; note here the switch to the complementary modulus. We may formulate this in terms of the fundamental half-periods $(\Omega_{\kappa}, \Omega_{\kappa}')$ for $P_{\kappa}$ and $(\omega_{\lambda}, \omega_{\lambda}')$ for $p_{\lambda}$: thus, 
$$\Omega_{\kappa} = - \ii \sqrt2 \omega_{\lambda}' \; \; {\rm and} \; \; \Omega_{\kappa}' = \ii \sqrt2 \omega_{\lambda}.$$ 

\medbreak 

Consider the periods of the elliptic function ${\rm dn}_2$. On the one hand, as we noted above, ${\rm dn}_2$ and $p = p_{\kappa}$ are coperiodic; in particular, ${\rm dn}_2$ has $(2 \omega_{\kappa}, 2 \omega_{\kappa}')$ as a fundamental pair of periods. On the other hand, we have the following result. 

\medbreak 

\begin{theorem} \label{dn2per}
The elliptic function ${\rm dn}_2$ has $(\Omega_{\kappa}, 2 \Omega_{\kappa}')$ as a fundamental pair of periods. 
\end{theorem} 

\begin{proof} 
From the definition ${\rm dn}_2 = 1 - 2 \rn^2$ it follows in the one direction that periods of $\rn$ are periods of ${\rm dn}_2$. In the other direction, if $\ppi$ is a period of ${\rm dn}_2$ then $\rn (z + \ppi)^2 = \rn (z)^2$ so that $\rn (z + \ppi) = \pm \rn (z)$ and therefore $\rn (z + 2 \ppi) = \rn (z)$; here, all of these equalities hold identically. Finally, Theorem \ref{plusOmega} implies that the half-period $\Omega_{\kappa}$ of $\rn$ is a period of ${\rm dn}_2$ while Theorem \ref{plusOmega'} prevents the incongruent half-periods of $\rn$ from being periods of ${\rm dn}_2$. 
\end{proof} 

\medbreak 

Thus 
$$\Omega_{\kappa} = 2 \omega_{\kappa} \; \; {\rm and} \; \; \Omega_{\kappa}' = \omega_{\kappa}'.$$ 

We are suddenly in a position to render the fundamental periods of the Weierstrass functions $P_{\kappa}$ and $p_{\kappa}$ in explicit hypergeometric terms. 

\medbreak 

\begin{theorem} \label{O}
The fundamental periods $(2 \Omega_{\kappa}, 2 \Omega_{\kappa}')$ of $P_{\kappa}$ are given by 
$$\Omega_{\kappa} = \pi \, F(\tfrac{1}{4}, \tfrac{3}{4} ; 1; \kappa^2)$$ 
$$\Omega_{\kappa}' = \ii \tfrac{1}{\sqrt2} \pi \, F(\tfrac{1}{4}, \tfrac{3}{4} ; 1; 1 - \kappa^2).$$ 
\end{theorem} 

\begin{proof} 
The formula for $\Omega_{\kappa}$ is already in evidence, as noted after Theorem \ref{P} in reference to Theorem \ref{4K}. The formula for $\Omega_{\kappa}'$ follows upon combining $\Omega_{\kappa}' = \ii \sqrt2 \omega_{\lambda}$ (as noted after Theorem \ref{Pp}) and $\omega_{\lambda} = \tfrac{1}{2} \Omega_{\lambda}$ (as noted after Theorem \ref{dn2per}). 
\end{proof} 

\medbreak 

The corresponding period ratio is therefore 
$$\frac{\Omega_{\kappa}'}{\Omega_{\kappa}} = \ii \frac{1}{\sqrt2} \, \frac{F(\tfrac{1}{4}, \tfrac{3}{4} ; 1; 1 - \kappa^2)}{F(\tfrac{1}{4}, \tfrac{3}{4} ; 1; \kappa^2)}.$$

\medbreak 

\begin{theorem} \label{o}
The fundamental periods $(2 \omega_{\kappa}, 2 \omega_{\kappa}')$ of $p_{\kappa}$ are given by
$$\omega_{\kappa} = \tfrac{1}{2} \pi \, F(\tfrac{1}{4}, \tfrac{3}{4} ; 1; \kappa^2)$$ 
$$\omega_{\kappa}' = \ii \tfrac{1}{\sqrt2} \pi \, F(\tfrac{1}{4}, \tfrac{3}{4} ; 1; 1 - \kappa^2).$$ 
\end{theorem} 

\begin{proof} 
A consequence of Theorem \ref{O} and the formulae that immediately precede it. Of course, these formulae for half-periods agree with the formulae for periods appearing in Theorem 3.3 of [2014]. 
\end{proof} 

\medbreak 

The corresponding period ratio is therefore 
$$\frac{\omega_{\kappa}'}{\omega_{\kappa}} = \ii \sqrt2 \, \frac{F(\tfrac{1}{4}, \tfrac{3}{4} ; 1; 1 - \kappa^2)}{F(\tfrac{1}{4}, \tfrac{3}{4} ; 1; \kappa^2)}.$$

\medbreak 

There are several ways to approach the function ${\rm cn}_2$. One approach has an algebraic feel: notice that 
$$\cos^2 \psi = 1 - \kappa^2 \sin^2 \phi = 1 - \kappa^2 + \kappa^2 \cos^2 \phi = \lambda^2 + \kappa^2 \cos^2 \phi$$ 
so that 
$$\kappa^2 \cos^2 \phi = {\rm dn}_2 ^2 - \lambda^2 = ({\rm dn}_2 - {\rm dn}_2 (\omega)) ({\rm dn}_2 - {\rm dn}_2 (\omega + \omega'))$$
on the real line. This last function is elliptic, with double zeros (represented by $\omega$ and $\omega + \omega'$) and quadruple poles (represented by $\omega'$); accordingly, it has two square-roots that are meromorphic and indeed elliptic. We may therefore define ${\rm cn}_2$ to be the square-root of $({\rm dn}_2^2 - \lambda^2)/\kappa^2$ with value $1$ at the origin. Thus: $\kappa^2 {\rm cn}_2 ^2 = {\rm dn}_2 ^2 - \lambda^2$ and ${\rm cn}_2 (0) = 1.$ 

\medbreak 

However, there is a more direct approach. Recall from the proof of Theorem \ref{IVP} that the functions $\rn : \R \to \R$ and $\cos \phi : \R \to \R$ satisfy $\rn ' = \tfrac{1}{2} \kappa \cos \phi$. This at once shows us that $\cos \phi$ extends to $\C$ as an elliptic function, which we call ${\rm cn}_2.$ 

\medbreak 

\begin{theorem} \label{cn2} 
The elliptic function ${\rm cn}_2$ defined by 
$$\tfrac{1}{2} \kappa \, {\rm cn}_2 = \rn '$$ 
satisfies 
$$\kappa^2 {\rm cn}_2 ^2 = {\rm dn}_2 ^2 - \lambda^2.$$

\end{theorem} 

\begin{proof} 
The relationship between ${\rm cn}_2$ and ${\rm dn}_2$ was exhibited ahead of the present Theorem; alternatively, it follows from Theorem \ref{IVP} by squaring, thus 
$$\kappa^2 {\rm cn}_2^2 = 4 (\rn ')^2 = 4 \rn^4 - 4 \rn^2 + \kappa^2 = (1 - 2 \rn^2)^2 - 1 + \kappa^2 = {\rm dn}_2 ^2 - \lambda^2.$$ 
\end{proof} 

\medbreak 

From among the various further properties of ${\rm cn}_2$ we select a few for mention. The identity ${\rm dn}_2^2 = \lambda^2 + \kappa^2 {\rm cn}_2^2$ has immediate consequences: it shows that ${\rm cn}_2$ and ${\rm dn}_2$ are copolar, both functions having poles at $\omega'$ (and points congruent); it also shows that ${\rm cn}_2$ has zeros precisely where ${\rm dn}_2 = \pm \lambda$ (at points congruent to $\omega$ and points congruent to $\omega+ \omega'$). Differentiation of the result in Theorem \ref{plusOmega} yields $\rn'(z + \Omega) = - \rn'(z)$ so that ${\rm cn}_2 (z + \Omega) = - {\rm cn}_2 (z)$; taking into account the observation that follows Theorem \ref{dn2per}, we see that ${\rm cn}_2$ has real period $4 \omega$ (twice the real period $2 \omega$ of ${\rm dn}_2$).

\medbreak 

It remains to consider the elliptic extendibility of the function ${\rm sn}_2 = \sin \phi$. We shall be brief: $\sin \phi$ {\it does not} extend to an elliptic function; in fact, more is true. 

\medbreak 

\begin{theorem} \label{snno}
The function $\sin \phi : \R \to \R$ lacks a meromorphic extension to $\C.$ 
\end{theorem} 

\begin{proof} 
Note that 
$$\kappa^2 \sin^2 \phi = \sin^2 \psi = 4 \sin^2 \tfrac{1}{2} \psi \cos^2 \tfrac{1}{2} \psi$$ 
whence 
$$\tfrac{1}{4} \kappa^2 \sin^2 \phi = \rn^2 (1 - \rn^2).$$ 
Consider the elliptic function on the right side of this equation: its zeros at points where $\rn = \pm 1$ are simple, because there $(\rn ')^2 = \tfrac{1}{4} \kappa^2$ (nonzero) by Theorem \ref{IVP}; so $\rn^2 (1 - \rn^2)$ has no square-roots that are meromorphic in the plane. 
\end{proof} 

\medbreak 

The function ${\rm sn}_2^2$ is of course elliptic; along with ${\rm cn}_2^2$ and ${\rm dn}_2^2$ it satisfies the familiar equations ${\rm cn}_2^2 + {\rm sn}_2^2 = 1$ and ${\rm dn}_2^2 + \kappa^2 {\rm sn}_2^2= 1.$ 

\medbreak 

\section*{Remarks} 

\medbreak 

We close our account with a couple of remarks. 

\medbreak 

In [2016] Shen presented a context within which to treat uniformly the elliptic theories in signatures three, four and six. In signature four, the relevant elliptic functions are solutions of the differential equation 
$$(y ')^2 = T_4 (y) - (1 - 2 \mu^2)$$
in which $T_4$ is the degree four Chebyshev polynomial of the first kind and $\mu$ is a parameter; this differential equation is supplemented by choosing the initial value $y(0)$ to be one of the zeros of the quartic $T_4 (y) - (1 - 2 \mu^2)$. The solution $y_4$ to the resulting initial value problem is then rendered explicitly in Weierstrassian terms by the method that is indicated in Theorem \ref{dn2} above. 

\medbreak 

Expanding the Chebyshev polynomial $T_4$ in full, the differential equation for $y_4$ reads 
$$(y ')^2 = 8 y^4 - 8 y^2 + 2 \mu^2.$$ 
When the parameter $\mu$ is taken to be the modulus $\kappa$, this is just the differential equation for $\rn$ modulo rescaling: explicitly, the rule $y(z) = \rn (\sqrt8 z)$ defines a solution of the differential equation 
$$(y ')^2 = 8 y^4 - 8 y^2 + 2 \kappa^2;$$
a suitable shift in the argument then secures the initial condition that singles out the function $y_4$. 

\medbreak 

In this way, the elliptic function $y_4$ of [2016] engenders the elliptic functions of [2014]. 

\medbreak 

Finally, a comment regarding the evaluation of periods in Theorem \ref{O} and Theorem \ref{o}.  It is perhaps worth noting that it is customary to perform integral calculations in order to make explicit these fundamental periods. Our approach made use of just one such calculation: namely, the calculation of the `complete integral' to which we alluded before Theorem \ref{4K}. 

\bigbreak

\begin{center} 
{\small R}{\footnotesize EFERENCES}
\end{center} 
\medbreak

[1927] E.T. Whittaker and G.N. Watson, {\it A Course of Modern Analysis}, Fourth Edition, Cambridge University Press. 

\medbreak 

[1953] A. Erdelyi (director), {\it Higher Transcendental Functions}, Volume 1, McGraw-Hill. 

\medbreak 

[1989] B.C. Berndt, {\it Ramanujan's Notebooks Part II}, Springer-Verlag. 

\medbreak 

[2014] Li-Chien Shen, {\it On a theory of elliptic functions based on the incomplete integral of the hypergeometric function $_2 F_1 (\frac{1}{4}, \frac{3}{4} ; \frac{1}{2} ; z)$}, Ramanujan Journal {\bf 34} 209-225. 

\medbreak 

[2016] Li-Chien Shen, {\it On Three Differential Equations Associated with Chebyshev Polynomials of Degrees 3, 4 and 6}, Acta Mathematica Sinica, English Series {\bf 33} (1) 21-36. 

\medbreak

\end{document}